\documentclass[12pt]{amsart}
\usepackage{amsmath, amsthm, mathrsfs, amssymb, verbatim, bbm, url, appendix}
\usepackage{amscd}
\usepackage{latexsym,amssymb}
\usepackage{graphicx}
\usepackage{color}
\usepackage{enumerate}
\allowdisplaybreaks[3]

\usepackage{latexsym}
\usepackage{epsfig}
\usepackage{amsfonts}
\usepackage{enumerate}
\usepackage{times}

\newcommand{\ra}{\right\rangle}
\newcommand{\la}{\left\langle}

\newcommand{\iy}{\mathcal{Y}}

\newtheorem{dilshift}{Remark}[section]
\newtheorem{conjkrr}{Conjecture}[section]
\newtheorem{virtualkrr}[conjkrr]{Proposition}
\newtheorem{in0}[conjkrr]{Remark}
\newtheorem{ramdatum}[conjkrr]{Definition}
\newtheorem{ex1}[conjkrr]{Example}
\newtheorem{in1}[conjkrr]{Remark}
\newtheorem{in2}[conjkrr]{Remark}
\newtheorem{in3}[conjkrr]{Remark}
\newtheorem{nonperm}[conjkrr]{Remark}
\newtheorem{in4}[conjkrr]{Proposition}
\newtheorem{in5}[conjkrr]{Corollary}
\newtheorem{acg}[conjkrr]{Proposition}
\newtheorem{acg1}[conjkrr]{Remark}

\newtheorem{mainst}[conjkrr]{Theorem}
\newtheorem{kcoh}[conjkrr]{Corollary}
\newtheorem{rem}[conjkrr]{Remark}

\begin{document}

\title{A formula for the total permutation-equivariant K-theoretic Gromov-Witten potential }
\author[Tonita]{Valentin Tonita}
\address{Institut f\"ur Mathematik\\ Rudower Chaussee 25 Raum 1425 \\ Berlin \\ Germany}
\email{valentin.tonita@hu-berlin.de}

\date{\today}

\begin{abstract}
We give a summation over graphs type formula for the permutation-equivariant K-theoretic Gromov-Witten total potential of a projective manifold
$X$ in terms of cohomological Gromov-Witten (cohGW) invariants of $X$. We achieve this by describing combinatorially the Kawasaki strata of the moduli spaces
of stable maps to $X$.
\end{abstract}

\maketitle

\section{Introduction}

   Let $X$ be a complex projective manifold. K-theoretic GW invariants are symplectic invariants associated to $X$ defined as Euler characteristics of certain sheaves on $X_{g,n,d}$, the moduli spaces of stable maps to $X$.  They have been introduced by Givental and Lee in \cite{giv_lee} and \cite{ypl}. They give rise to similar structures with the better studied cohomological GW invariants, for e.g. one can use the genus zero invariants to define a quantum K-product deforming the usual tensor product of coherent sheaves on $X$. At the same time, K-theoretic GW invariants are more difficult to compute than the cohomological ones.
 
  This problem was addressed in \cite{gito} in genus zero, where it was shown how, in principle , one can determine the K-theoretic GW invariants of $X$ from the cohGW invariants. The starting point of the proof is the Kawasaki Riemann Roch (KRR) formula applied to the moduli spaces of genus zero stable maps to $X$. KRR formula expresses Euler characteristics on a compact complex orbifold as cohomological integrals on its inertia orbifold. Roughly speaking, the inertia orbifold parametrizes strata of maps with non-trivial symmetries.
  An essential 
  part of the proof in \cite{gito} is the description of the connected components of the inertia orbifold of the moduli spaces $X_{0,n,d}$ in terms of combinatorial data. This is not possible to achieve for higher genus due to the presence of discrete automorphisms of the domain curve. 
  
   More recently Givental introduced in \cite{giv} {\em permutation-equivariant K-theoretic GW invariants}: these are Euler characteristics on the orbifolds $X_{g,n,d}/S_n$, where the $S_n$ action renumbers the marked points. The permutation-equivariant version of the theory seems to fit better within the broader framework of mirror symmetry. Givental completely characterized the genus zero potential of $X=[pt.]$ and he conjectured that the 
   total potential is a solution of a finite difference version of the KdV equation.
  
   In the present paper we make a first step towards this conjecture by expressing higher genus permutation-equivariant invariants in terms of cohomological ones. Again the starting point is KRR formula applied to the moduli spaces $X_{g,n,d}/S_n$. It turns out that strata of maps with symmetries can be described by combinatorial data: graphs where vertices correspond to irreducible components of the domain curve and edges correspond to nodes. The vertices of the graphs parametrize stable maps to the target orbifold $X\times B\mathbb{Z}_m$ for some integer $m$. The contributions associated to each vertex are shown to be cohomological integrals which can be explicitly expressed in terms of the cohGW invariants of $X$. To the edges one can associate differential operators which act on the adjacent vertex contributions.
    
    Our main result expresses the permutation-equivariant K-theoretic total potential of $X$ as a summation over graphs of integrals expressible in terms of cohGW of $X$. Our results are valid for $X=[pt.]$ and for general $X$ modulo Conjecture \ref{conj1}.
    
    The paper is organized as follows. In Section 2 we define the main objects of interest: permutation-equivariant K-theoretic Gromov-Witten invariants and the total K-theoretic potential of $X$. In Section 3 we state KRR formula and we encode strata of maps with non-trivial symmetries in $X_{g,n,d}$ by combinatorial data: graphs with additional "ramification data" attached to each vertex. In Section 4 we identify the vertex contributions as cohomological integrals on the moduli spaces of stable maps to $X\times B\mathbb{Z}_m$ and we process them via twisted GW theory to express them in terms of cohGW invariants of $X$. In the last section we describe the edge operators and we state our main result Theorem \ref{mainthm}.
  
   \textbf{Acknowledgments}. I would like to thank A. Givental for useful discussions and for explaining to me his work \cite{giv}.
\section{Permutation-equivariant K-theoretic Gromov-Witten invariants}

 Let $X$ be a complex projective manifold. We denote by $X_{g,n,d}$ Kontsevich's moduli stack of genus zero stable maps to $X$: they parametrize
  data $(C, x_1,\ldots ,x_n, f)$  such that
  \begin{itemize}
  \item $C$ is a connected projective complex curve of arithmetic genus $g$ with at most nodal singularities.
  \item  $(x_1,\ldots x_n) \in C$ is an ordered $n$-tuple of distinct smooth points on $C$ (they are called marked points).
  \item $f:C\to X$ is a map of degree $d\in H_2(X,\mathbb{Z}$).
  \item The data $(C, x_1,\ldots ,x_n, f)$ has finite automorphism group, where an automorphism is defined to be an automorphism $\varphi:C\to C$ such 
  that $\varphi(x_i) =x_i$ for all $i=1,..n$ and $f\circ \varphi =f$. 
  \end{itemize}
   For each $i=1,..,n$ there are evaluation maps $ev_i : X_{g,n,d}\to X $ defined by sending a point $(C, x_1,\ldots ,x_n, f)\mapsto f(x_i)$ and   cotangent line bundles $L_i \to X_{g,n,d}$ whose fibers over a point $(C, x_1,\ldots ,x_n, f)$ are identified with $T^\vee_{x_i}C$.
  
 Denote by $K^0(X)$ he topological ring of vector bundles on $X$ and let $E_i\in K^0(X)$. K-theoretic Gromov-Witten (GW) invariants have been defined by Givental and Lee (\cite{giv_lee},\cite{ypl}) as sheaf holomorphic Euler characteristics on $X_{g,n,d}$  obtained using the classes $ev^*_i(E_i)$ and the line bundles $L_i$:
   \begin{align*}
   \chi\left(X_{g,n,d}, \mathcal{O}_{g,n,d}^{vir}\otimes _{i=1}^n ev_i^* (E_i)L_i^{k_i}\right)\in \mathbb{Z}.
   \end{align*}
 
 Here $\mathcal{O}_{g,n,d}^{vir}\in K_0(X_{g,n,d})$ is the virtual structure sheaf defined in \cite{ypl}.
  We will use correlator notation for the invariants:
   \begin{align*}
   \la E_1 L^{k_1},\ldots , E_n L^{k_n}\ra_{g,n,d}:= \chi\left(X_{g,n,d}, \mathcal{O}_{g,n,d}^{vir}\otimes _{i=1}^n ev_i^* (E_i)L_i^{k_i}\right).
   \end{align*}

 Givental has recently generalized  the definition of K-theoretic GW invariants by considering sheaf Euler characteristics on the Deligne-Mumford stacks $X_{g,n,d}/S_n$. The action of $S_n$ on the moduli spaces is given by renumbering the marked points. 
 In other words the alternating sum
       \begin{align*}
       \left[\mathbf{t}(L),\ldots ,\mathbf{t}(L)\right]_{g,n,d}:= \sum (-1)^m H^m\left(X_{g,n,d}; \mathcal{O}_{g,n,d}^{vir}\otimes_{i=1}^n \mathbf{t}(L_i)\right)
       \end{align*}
       
     carries a well-defined $S_n$ module structure. Here $\mathbf{t}(q)$ is a Laurent polynomial in $q$ with coefficients in $K(X)\otimes \Lambda$, where $\Lambda$ is an algebra which is assumed to carry 
     $\psi^k$ operations and, for convergence purposes, to have a maximal ideal $\Lambda_+$. We endow $\Lambda$ with the corresponding $\Lambda_+$-adic topology. In applications $\Lambda$ naturally satisfies these assumptions: it usually includes the Novikov variables, the algebra of symmetric polynomials in a given number of variables and/or the torus equivariant $K$-ring of the point.
     For suitable choices of $\Lambda$ the permutation-equivariant invariants encode all the information about the $S_n$ modules above. We refer to \cite{giv} for details.
      
       The invariants :
      \begin{align*}
         \la \mathbf{t}(L),\ldots ,\mathbf{t}(L) \ra^{S_n}_{g,n,d}
         \end{align*}
         are defined as K-theoretic push forwards of the classes  $\mathcal{O}_{g,n,d}^{vir}\otimes_{i=1}^n \mathbf{t}(L_i)$ along the map  $X_{g,n,d}/S_n \to [pt.]$. 
      Let $$ \mathcal{K}_+ := \mathbb{C}[q,q^{-1}]\otimes K^0(X, \Lambda_+),$$
 where we assume $\Lambda$ contains the Novikov ring based on degrees of effective curves in $X$ denoted $\mathbb{C}[[Q]]$. 
       The genus $g$ potential of $X$ is defined as 
        $$ \mathcal{F}_g^K: \mathcal{K}_+ \to \Lambda_+,$$
       $$\mathcal{F}^K_g(\mathbf{t}(q)+1-q) =\sum_{n,d}Q^d  \la \mathbf{t}(L),\ldots ,\mathbf{t}(L)\ra^{S_n}_{g,n,d} .$$ 
 The total potential is
     $$\mathcal{D}_X^K :\mathcal{K}_+ \to \Lambda_+[[\hbar, \hbar^{-1}]],$$
       $$ \mathcal{D}_X^K(\mathbf{t}(q)+1-q): = \operatorname{exp}\left(\sum_{g\geq 0}\hbar^{g-1}\mathcal{F}^K_g\right).$$

   \begin{dilshift}
  {\em The translation by $1-q$ in the argument of the potentials is called the dilaton shift.} 
   \end{dilshift}
   
\section{Kawasaki strata of $X_{g,n,d}$}

To express holomorphic Euler characteristics of a vector bundle $V$ on a compact complex orbifold $\iy$ as a cohomological integral one uses Kawasaki Riemann Roch (KRR) theorem of \cite{kawasaki} (proven by T\"oen in \cite{toen}  for proper smooth Deligne-Mumford stacks). The integrals are supported on the inertia orbifold $I\iy$ of $\iy$:
       \begin{align}
        \chi(\iy, V) = \sum_\mu \int_{\iy_\mu} \operatorname{Td}(T_{\iy_\mu})\operatorname{ch} \left(\frac{\operatorname{Tr}(V)}{\operatorname{Tr}(\Lambda^\bullet N^\vee_\mu)}\right).\label{eqn:KRR} 
        \end{align}
    We now explain this ingredients of this formula. $I\iy$ is the inertia orbifold of $\iy$, given set-theoretically by pairs $(y, (g))$, where $y\in \iy$ and $(g)$ is (the conjugacy class of) a symmetry which fixes $y$.  We denote by $\iy_\mu$ the connected components of $I\iy$.

        For a vector bundle $V$, let $V^\vee$ be the dual bundle to $V$. The restriction of $V$ to $\iy_\mu$ decomposes in characters of the $g$ action. Let $V_r^{(l)}$ be the subbundle of the restriction of $V$ to $\iy_\mu$ on which $g$ acts with eigenvalue $e^{\frac{2\pi i l}{r}}$. Then the  trace $\operatorname{Tr}(V)$ is defined to be the orbibundle whose fiber over the point $(p, (g))$ of $\iy_\mu$ is 
         \begin{align*}
      \operatorname{Tr}(V):= \sum_{0\leq l\leq r-1} e^{\frac{2\pi i l}{r}} V^{(l)}_r . 
         \end{align*} 
      Finally, $\Lambda^\bullet N^\vee_\mu$ is the K-theoretic Euler class of the normal bundle $N_\mu$ of $\iy_\mu$ in $\iy$. $\operatorname{Tr}(\Lambda^\bullet N^\vee_\mu)$ is invertible because the symmetry $g$ acts with eigenvalues different from $1$ on the normal bundle to the fixed point locus. 
 
    We call the connected components of $I\iy$ {\em Kawasaki strata}. 
    
    $X_{g,n,d}$ is not smooth but it has a perfect obstruction theory which can be used to define its virtual fundamental class (see \cite{behfan}). For a stack $(\mathcal{Y},E^\bullet)$ with a perfect obstruction theory one can choose an explicit resolution of $E^\bullet$ as a complex of vector bundles $E^{-1}\to E^0$. Let $E_0\to E_1$ be the dual complex. Then the virtual tangent bundle of $\mathcal{Y}$ can be defined as the class $[E_0] \ominus [E_1] \in K^0(\mathcal{Y})$.
    Moreover the Kawasaki strata of $\mathcal{Y}$ inherit perfect obstruction theories which can be used to define their virtual normal bundles.
     
    It was proved in \cite{to_k} that one can apply KRR theorem to the moduli spaces $X_{0,n,d}$ by replacing all the ingredients in the formula with their virtual counterparts.

\begin{conjkrr} \label{conj1}
{\em The virtual KRR formula of \cite{to_k} holds for the moduli spaces $X_{g,n,d}$}.
\end{conjkrr}
 
 Conjecture \ref{conj1} is trivially true for the case $X=[pt.]$. Moreover it is easy to prove the following
  \begin{virtualkrr}
  {\em If conjecture \ref{conj1} is true for $X_{g,0,d}$ then it is true for $X_{g,n,d}/S_n$. } 
  \end{virtualkrr}
 \begin{proof}
 The map $f: X_{g,n,d}/S_n \to X_{g,0,d}$ which forgets the $n$ marked points is $S_n$-equivariant and the virtual structure sheaf on $X_{g,n,d}/S_n$ is the pullback of the virtual structure sheaf on $X_{g,0,d}$. Hence for any $S_n$-equivariant class $\mathcal{E}$ we have
  $$ \chi(X_{g,n,d}/S_n , \mathcal{O}_{g,n,d}^{vir}\otimes \mathcal{E}) =\chi(X_{g,0,d}, \mathcal{O}_{g,0,d}^{vir}\otimes f_*\mathcal{E} ).$$
 According to our assumption one can apply virtual KRR for $X_{g,0,d}$ to express the RHS above as a cohomological integral on $IX_{g,0,d}$. The integrands will involve the class $\operatorname{ch} \circ \operatorname{Tr}(f_*\mathcal{E})$. This can be processed using T\"oen's Riemann-Roch (\cite{tseng}, Appendix A) and rewritten
   as $$(If)_* \left[\operatorname{ch}\circ \operatorname{Tr}(\mathcal{E})\widetilde{\operatorname{Td}}(T_f)\right].$$ 

    We refer the reader to \cite{tseng}(Appendix A) for the definition of the class $\widetilde{\operatorname{Td}}$.  It is straightforward to check that using these two steps one expresses the Euler characteristic of the class $\mathcal{O}_{g,n,d}^{vir}\otimes \mathcal{E} $ as a cohomological integral on the inertia orbifold of $X_{g,n,d}/S_n$ of exactly the form in the virtual KRR theorem. \end{proof}   
 Let us now describe the Kawasaki strata of $ Z_\gamma\subset X_{g,n,d}/S_n$ fixed by a symmetry $\gamma$. We encode the combinatorial data of a {\em generic} stable map $C_\gamma$ in $Z_\gamma$  by a graph $\Gamma$: 
 
  - the vertices $v$ of $\Gamma$ correspond to irreducible components of the quotient curve $C_\gamma/\langle \gamma\rangle$. We denote by $C_v$ the preimage of each such irreducible component on the cover curve $C_\gamma$ and by $m_v$ the smallest integer such that $\gamma^{m_v}$ is the identity when restricted to $C_v$. The restriction of the stable map to $C_v$ can be identified with a map $C_v / \mathbb{Z}_{m_v}\to X\times B\mathbb{Z}_{m_v}$. We denote by $g_v$  the genus of the quotient curve $C_v/\mathbb{Z}_{m_v}$ and $d_v\in H_2(X,\mathbb{Z})$ the degree of the (quotient) map.
 
  - for every node connecting the curves $C_{v_i}/\mathbb{Z}_{m_i}$ and $C_{v_j}/\mathbb{Z}_{m_j}$ there is an edge between $v_i$ and $v_j$.  
  
  \begin{in0}
   {\em  Notice that edges correspond to nodes which cannot be smoothed without breaking the symmetry. Inside $Z_\gamma$ there are curves with more nodes, namely the ones where the same cyclic group acts on both branches in a balanced way (i.e. with eigenvalues inverse to each other on the cotangent space at the node to each branch).}
   \end{in0}
  In addition we record the following ramification data attached to vertices $v$. Consider a (total or partial) branch point on the quotient curve and assume it has $d$ preimages on the cover. There is an action of a cyclic subgroup of $\mathbb{Z}_{m_v}:=\langle \gamma_v\rangle$ of order $m_v/g.c.d.(d,m_v)$ on the cotangent space at each of the ramification points above the given branch point.  
   Following \cite{np} (see Definition 2.8, Example 2.11) we identify the set of pairs $(H,\chi) $ where $H$ is a nontrivial subgroup of $\mathbb{Z}_{m_v} $ and $\chi$ an injective character  with the set $\{1,2,\ldots m_v-1\}$: the element $r$ of the set corresponds to the subgroup $\langle \gamma_v^r\rangle \leq \mathbb{Z}_{m_v}$ and the character given by 
  \begin{align*}
  \gamma_v^r \mapsto e^{2\pi i \frac{gcd (r, m_v)}{m_v}} . 
  \end{align*}
  We refer to such a branch point as a {\em branch point of type $r$}.
  
  \begin{ex1}
  {\em Take $m_v=12$. A branch point of type $8$ has $g.c.d. (8,12) = 4$ preimages on the covercurve and the element $8\in \mathbb{Z}_{12}$ acts on each of the cotangent spaces by multiplication by $e^{2\pi i /3}$}.
  \end{ex1}
   
  So we record two strings of positive integers $(a_1, \ldots a_{m_v-1}, b_1,\ldots b_{m_v-1})$ where $b_r$ is the number of branch points of type $r$ and $a_r\leq b_r$ is the number of marked points. In addition let $a_0$ the number of marked points on $C_v/\mathbb{Z}_{m_v}$ which are unramified. Each such a marked point has as preimage in $C_v$ an $m_v$-tuple of marked points which get permuted cyclically by the symmetry $\gamma_v$.  
  \begin{ramdatum}
  {\em We call }
  \begin{align*}
  \vec{R}_v := (a_0, \ldots a_{m_v-1}, b_1,\ldots , b_{m_v-1})
  \end{align*}
   the ramification data {\em associated to the vertex $v$}.
  \end{ramdatum}

  We also need to keep track of the ramification type of the node.  To do this we define a flag to be a pair $F=(v,e)$ with $v$ adjacent  to $e$. To each flag $(v,e)$ we associate an integer $1\leq r_F \leq m_v$: the ramification type of the node viewed as a point on $C_v/\mathbb{Z}_{m_v}$. We adopt the convention that if $r_F=m_v$ there are $m_v$ preimages of the node on the cover curve. 
  
  The ramification data is subject to the following constraint (see \cite{np})
  \begin{align*}
  m_v |\sum_{r=1}^{m_v-1} rb_r  + \sum_{F=(v,e)}r_F\quad .
  \end{align*}
    
  \begin{in1}
  {\em The ramification data determines via Riemann Hurwitz formula the genus $g(C_v)$ of $C_v$ (recall that $C_v$ is irreducible for a generic element of $Z_\gamma$):}
   \begin{align}
   2g(C_v)-2 = m_v(2g_v-2)+ \sum_{r=1}^{m_v-1} b_r (m_v- g.c.d.(i,m_v)) + \sum_{ v\in F}(m_v-g.c.d.(r_F,m_v)),
   \end{align}
  \end{in1}

  \begin{in2}
 {\em  The moduli spaces of stable maps to $X\times B\mathbb{Z}_m$ comes with evaluation maps at each marked point which land in the inertia stack $I(X \times B\mathbb{Z}_{m_v})$. $I(X\times B\mathbb{Z}_{m_v})$ consists of $m_v$ copies of $X\times B\mathbb{Z}_{m_v}$ labeled by elements of $\mathbb{Z}_{m_v}$ (powers of  $\gamma_v$). The ramification data determines the connected components of $I(X \times B\mathbb{Z}_{m_v})$ where the evaluation map at each marked point, branch point or node (viewed as a point on $C_v/\mathbb{Z}_{m_v}$ after separating the branches) lands.}
  \end{in2}
  
 \begin{in3}
 {\em The preimage of a node on $C_v$ may consist of more than a single point - in fact it consists of exactly g.c.d. $(r_F, m_v$) points. Let $F=(v,e)$ and call $w$ the other vertex adjacent to $e$. On the cover curve $C_v$ each preimage of the node is again a node connecting $C_v$ with $C_w$. In this case $C_w$ is a disjoint union of (necessarily isomorphic) curves which are permuted cyclically by $\gamma_v\in \mathbb{Z}_{m_v}$ }.  
 \end{in3}
 
   \begin{nonperm}
   {\em The data above can easily be adjusted to describe Kawasaki strata in $X_{g,n,d}$: requiring each marked points to be fixed by symmetries $\gamma_v$ means that  $a_r=0$ if $(r, m_v)\neq 1$. Moreover  curves $C_w$ consisting of several components permuted cyclically are not allowed to carry marked points. }
   \end{nonperm}
 Let now $S_{A_r}, S_{B_r}$ be the permutation groups permuting the $a_r$ marked ramification points of type $r$ , respectively $b_r-a_r$ unmarked ramification points of the same type and denote by $S_v:= \prod_r S_{A_r}\prod_r S_{B_r}$. Moreover let $\vert b\vert=\sum_r b_r + $ the number of nodes on $C_v/\mathbb{Z}_{m_v}$.
 
  \begin{in4}
  Each vertex  parametrizes a closed and open substack of the stack of stable maps  to $(X\times B\mathbb{Z}_{m_v})_{g_v,\vert b\vert+a_0,d_v}/S_v$.  
  \end{in4}
  \begin{proof}
 Consider the case when the graph $\Gamma$ consists of only one vertex $v$, i.e. the generic domain curve in the Kawasaki stratum is irreducible. A map to $X\times B\mathbb{Z}_{m_v}$ consists of a map to $X$ together with a $\mathbb{Z}_{m_v}$ cover of the domain curve, branched over the marked points in a way encoded by the ramification data. It is easy to see that the proofs in \cite{np}, \cite{npot} (where $X=[pt.]$, and there is no $S_n$ action) carry on in our situation. Likewise, the generalization to other graphs is immediate: the graph records the combinatorial type of a generic domain curve in a given Kawasaki stratum and each vertex together with its ramification data specifies a cyclic cover of  the domain curve.  \end{proof}
  \begin{in5}
  {\em The disjoint union $\coprod_{g,n,d}I(X_{g,n,d}/S_n)$ is parametrized by graphs $\Gamma$ with ramification data attached to the vertices subject to the constraints described above.}
  \end{in5}
 
  The K-theoretic potential of $X$ is a summation over graphs $\Gamma$  of (cohomological) integrals. In the following sections we express these integrals in terms of the cohGW potentials of $X$.
  \section{The vertex contributions}
 
 As hinted at in the previous section, the vertices of graphs $\Gamma$ contribute in KRR formula as certain GW invariants of $X\times B\mathbb{Z}_m$ (we will ignore the subscript $v$ throughout this section) additionally twisted by certain characteristic classes of the tangent and normal bundle of these strata. We will now express these contributions in terms of the cohGW theory of $X$. 
 
 Recall that the coh GW invariants of $X$ are defined as
 $$\langle \varphi_1 \psi^{k_1},\ldots ,\varphi_n \psi^{k_n}\rangle_{g,n,d}:= \int_{[X_{g,n,d}]}\prod_{i=1}^n ev_i^*(\varphi_i)\psi_i^{k_i},$$ 
where $\varphi_i\in H^*(X,\Lambda)$.
  Let $$\mathcal{H}_+:=\mathbb{C}[z^{-1},z]]\otimes H^*(X,\Lambda).$$
 Let $\mathcal{F}_g^X, \mathcal{D}_X$ the genus $g$, respectively total cohomological descendant potential of $X$. They are defined as 
 
 $$\mathcal{F}_g^X: \mathcal{H}_+ \to \Lambda,\quad \mathcal{F}_g^X (\mathbf{t}(z)-z)= \sum_{n,d}\frac{Q^d}{n!}  \la \mathbf{t}(\psi),\ldots ,\mathbf{t}(\psi)\ra^{X}_{g,n,d}, $$
 $$ \mathcal{D}_X:\mathcal{H}_+ \to \Lambda[[\hbar,\hbar^{-1}]],\quad \mathcal{D}_X(\mathbf{t}(z)-z) := \operatorname{exp}\left(\sum_{g\geq 0}\hbar^{g-1}\mathcal{F}_g^X\right).$$
   
   Recall that $H^*(X/\mathbb{Z}_m)=\oplus_{h\in \mathbb{Z}_m}H^*(X)$. Hence we can consider the domain of the total cohomological potential of $X\times B\mathbb{Z}_m$ as the direct sum $\oplus_{h\in \mathbb{Z}_m} \mathcal{H}_+$.
    According to \cite{jaki} the  cohomological potentials of $X$ and $X/\mathbb{Z}_m$ are related as follows: let $$\mathbf{t} =\sum_{h\in \mathbb{Z}_m} \mathbf{t}_h h , \quad \mathbf{t}_h\in \mathcal{H}_+.$$  
     Let $\chi_1,...,\chi_m$ be the $m$ characters of the irreducible one dimensional representations of $\mathbb{Z}_m$. Then:
  \begin{align}
 \mathcal{D}_{X/\mathbb{Z}_m}\left(\sum_h \mathbf{t}_h h - z\right) & = \prod_{j=1}^m \operatorname{exp} \left(\sum_{g\geq 0} \frac{\hbar^{g-1}}{m^{2-2g}}\mathcal{F}_g^X(\sum_h \mathbf{t}_h \chi_j(h) - z) \right) \nonumber\\
                                        & = \prod_{j=1}^m  \mathcal{D}_X \left(\frac{\sum_h \mathbf{t}_h \chi_j(h)-z}{m} \right) .\label{xxzm}
\end{align} 

  In the last equality we used the homogenity of the genus $g$ potential.
  
  Recall that KRR formula involves various characteristic classes of tangent and normal bundles of Kawasaki strata. We now describe these classes and use the formalism of twisted GW theory to express the generating series of integrals involving these classes (the twisted potential) in terms of $\mathcal{D}_{X/\mathbb{Z}_m}$ (the untwisted potential). 
  
  According to \cite{tom} the virtual tangent bundle $T_{g,n,d}$ to the moduli spaces $X_{g,n,d}$ can be described in $K_0(X_{g,n,d} )$ in terms of pushforwards along the universal curve $\pi:X_{g,n+1,d}\to X_{g,n,d}$: 
   \begin{align*}
  T_{g,n,d} =\pi_*(ev^* (T_X -1)) - \pi_*(L^{-1}_{n+1} -1)-(\pi_* i_* \mathcal{O}_\mathcal{Z})^\vee ,
   \end{align*}  
  where $\mathcal{Z}\subset X_{g,n+1,d}$ is the codimension two locus of nodes.
    
 Given a Kawasaki stratum in $Z\subset X_{g,n,d}$ which parametrizes maps with a $\mathbb{Z}_m$ symmetry we denote by $\pi : \mathcal{C}\to Z $ the universal family over $Z$ and by $\widetilde{\pi}:\mathcal{C}/\mathbb{Z}_m\to Z$ the  universal family of quotient curves by the action of $\mathbb{Z}_m$.  

We want to compute the trace $\operatorname{Tr}$ of $\gamma$ on $T_{g,n,d}$ in terms of the universal family $\widetilde{\pi}$. This is completely analogous with the case  $g=0$, which was dealt with in \cite{gito}. We write down the answer below.
Let $\mathbb{C}_{\zeta^k}$ be the line bundle on $X\times B\mathbb{Z}_m$ which is topologically trivial and on which $\gamma$ acts with eigenvalue $\zeta^k$. The first two summands in $T_{g,n,d}$ give

\begin{align}
\operatorname{Tr}(\pi_*(ev^* (T_X -1))= \sum_{k=0}^{m-1}\zeta^{-k}\widetilde{\pi}_*(ev^*(T_X-1)\otimes \mathbb{C}_{\zeta^k}),\label{indextwist}\\
 \operatorname{Tr}(\pi_*(L_{n+1}^{-1} -1)) =  \sum_{k=0}^{m-1}\zeta^{-k}\widetilde{\pi}_*(L_{n+1}^{-1} -1)\otimes \mathbb{C}_{\zeta^k}).\label{diltwist}
\end{align}
To compute the trace of $\gamma$ on the nodal class $(\pi_* i_* \mathcal{O}_\mathcal{Z})^\vee$  we need to look at the number of preimages of a node on the cover curve : denote by $\mathcal{Z}_{d}$ the nodal locus where the evaluation map $ev_+ \times ev_- $ at a node land in sectors of the inertia stack $IB\mathbb{Z}_m$ indexed by elements $h, h^{-1}$ of order $d:=ord(h)$. Then the node has $m/d$ preimages on the cover curve $C$. The smoothing bundle has dimension $m/d$: it contains a one dimensional subbundle which is tangent to the stratum and a subbundle of dimension $m/d -1$ normal to it.
We have
\begin{align}
\operatorname{Tr}(\pi_*i_*\mathcal{O}_{\mathcal{Z}_d})^\vee =\sum_{k=1}^{m/d-1}\zeta^{-k}(\widetilde{\pi}_*i_*(\mathcal{O}_{\mathcal{Z}_d} \otimes ev^*\mathbb{C}_{\zeta^k}))^\vee \label{nodaltwist}.
 \end{align}
 
 We now briefly recall the formalism of twisted GW theory, referring the reader to \cite{coatesgiv},\cite{tseng}, \cite{to1} for details.
 The correlators of a twisted theory are obtained by inserting in the integrals multiplicative characteristic classes of push-forwards along the universal family. For the target orbifold of our interest they take the form 
                         
                        $$ \int_{[(X/\mathbb{Z}_m)_{0,n,d}]}\left( \prod_{m=1}^n ev_m^*(\varphi_m)\psi^{k_m}_m\prod_{i} \mathcal{A}_i(\widetilde{\pi}_* (ev_{n+1}^* E)) \prod_j\mathcal{B}_j(\widetilde{\pi}_* [F(L^\vee_{n+1})-F(1)])\prod_k\mathcal{C}_k(\widetilde{\pi}_* i_*\mathcal{O}_\mathcal{Z}) \right),$$
                        where $\mathcal{A}_i, \mathcal{B}_j, \mathcal{C}_k$ are a finite number of multiplicative characteristic classes and $F$ is a polynomial with coefficients in $K^0(X)$. 
 
 One can pack the correlators in generators series to define the potential of the twisted theory. The formalism expresses the twisted potential in terms of the untwisted one. This was done in the most general setting in \cite{to1}. 
 
  Let us now describe the twisted theory we consider: for each root of unity $\zeta$, define the multiplicative characterstic class $\operatorname{Td}_\zeta$ whose value on line bundles is
   \begin{align*}
   \operatorname{Td}_\zeta (L)= \frac{1}{1-\zeta e^{-c_1(L)}}.
   \end{align*}
  We denote  by $\mathcal{D}_{tw}^{X/\mathbb{Z}_m}$ the potential of the cohGW theory of $X\times B\mathbb{Z}_m$ twisted by the following classes:

  \begin{align}
  \operatorname{Td}(\widetilde{\pi}_*(ev^* (T_X -1))\prod_{k=1}^{m-1}Td_{\zeta^k}\widetilde{\pi}_*(ev^*(T_X-1)\otimes \mathbb{C}_{\zeta^k}),\label{typea}\\
  \operatorname{Td}\widetilde{\pi}_*(1- L^{-1}_{n+1} )\prod_{k=1}^{m-1}\operatorname{Td}_{\zeta^k}\widetilde{\pi}_*((1- L_{n+1}^{-1} )\otimes ev^*\mathbb{C}_{\zeta^k})), \label{typeb}\\
  \prod_d\operatorname{Td}^\vee(-\widetilde{\pi}_*i_*\mathcal{O}_{\mathcal{Z}_d})\prod_{k=1}^{m/d-1}\operatorname{Td}^\vee_{\zeta^k}(-\widetilde{\pi}_*i_*(\mathcal{O}_{\mathcal{Z}_d} \otimes ev^*\mathbb{C}_{\zeta^k}) ).\label{typec}
   \end{align}
  This choice of classes is essentially imposed on us by formulae (\ref{indextwist}),(\ref{diltwist}),(\ref{nodaltwist}).  The classes $\operatorname{Td}$ (we use the convention $\operatorname{Td}^\vee (L) = \operatorname{Td}(L^\vee)$) are contributions of the Todd class of the tangent bundle in KRR, while the classes $\operatorname{Td}_\zeta$ occur in the denominator of KRR involving the normal bundle to Kawasaki strata.

   We now proceed to express $\mathcal{D}_{tw}^{X/\mathbb{Z}_m}$ in terms of the untwisted potential $\mathcal{D}_{X/\mathbb{Z}_m}$: let us denote by $\mathcal{D}_\mathcal{A}^{X/\mathbb{Z}_m}$ the potential of $X\times B\mathbb{Z}_m$ twisted by the classes (\ref{typea}).  Tseng's theorem in \cite{tseng} expresses  $\mathcal{D}_\mathcal{A}^{X/\mathbb{Z}_m}$  in terms of $\mathcal{D}_{X/\mathbb{Z}_m}$. The graph of the differential of the genus zero twisted potential is related to the untwisted one by a symplectic linear trasformation of $T^*\mathcal{H}_+$. To this symplectic operator one can associate a quantized operator via the quantization formalism introduced in \cite{gi}. Then  $\mathcal{D}_\mathcal{A}^{X/\mathbb{Z}_m}$ is obtained from $\mathcal{D}_{X/\mathbb{Z}_m}$ by applying this quantized operator. 
 
It turns out that in our case the symplectic operator of the theory twisted by the datum  $(\operatorname{Td}_{\zeta^k},(T_X -1)\otimes \mathbb{C}_{\zeta^k} )$ ($1\leq k\leq m-1$) takes the form of a product of symplectic operators $\prod_{h\in \mathbb{Z}_m} \triangle_{k,h}$ . 
In the following we will write these symplectic operators as Euler MacLaurin asymptotics of infinite products.
 The Euler MacLaurin asymptotics of a product $\prod_{l=1}^\infty e^{s(x-lz+\frac{k}{m} z)}$ is a formal power series defined by
  \begin{align*}
  \sum_{l>0}s(x+ \frac{k}{m}  z-lz)=\sum_{l\geq 0}B_l(\frac{k}{m})(z\partial_x)^{l-1}\frac{s(x)}{l!},
  \end{align*}
 where the Bernoulli polynomials are defined by 
 \begin{align*}
                    \sum_{l\geq 0}B_l(x)\frac{t^l}{l!}= \frac{te^{tx}}{e^t-1}.
                    \end{align*} 
 Let 
 $$ \triangle := \prod_i\prod_{l=1}^\infty \frac{x_i}{1-e^{-x_i+lz}}, $$
  where $x_i$ are the Chern roots of $T_X$ and they act by multiplication on $\mathcal{H}_+$. For $h=1\in \mathbb{Z}_m$,  
$\triangle_{k,1}$ is the Euler-Maclaurin expansion  of
 \begin{align*}
 \prod_i\prod_{l=1}^\infty \frac{1}{1-\zeta^ke^{-x_i+lz}}.
 \end{align*}

 For $h=\gamma^r \neq 1$ the operator $\triangle_{k,\gamma^r}$ is the Euler-Maclaurin asymptotics of 
  \begin{align*}
  \prod_i \prod_{l=1}^\infty \frac{1}{1-\zeta^k e^{-x_i +lz- \langle kr/m\rangle z }},
  \end{align*}
  where $\langle t \rangle$ is the fractionary part of $t$ .

The quantization of a linear symplectomorphism $A$ of $T^*\mathcal{H}_+$ is obtained as follows. Let $H$ be the infinitesimal symplectic transformation
$\operatorname{ln}(A)$. To $H$ one can naturally associate a quadratic hamiltonian 
 $$h:=\frac{1}{2}\Omega(\mathbf{t}(z), H\mathbf{t}(z)),$$
 where $\Omega$ is the symplectic form. The quantization of quadratic hamiltonians is defined in a Darboux basis $\{p_i,q_i\}$ by
 $$\widehat{p_ip_j} = \hbar \partial_i \partial_j,\quad \widehat{p_iq_j} = \partial_i q_j,\quad \widehat{q_iq_j}=\frac{q_iq_j}{\hbar}.  $$
 Then one defines
 $$\hat{A}:=\operatorname{exp}(\hat{h}).$$
The potential $\mathcal{D}_{\mathcal{A}}^{X/\mathbb{Z}_m}$ is given by (see \cite{to1}, Corollary 6.1)
\begin{align}
\mathcal{D}_{\mathcal{A}}^{X/\mathbb{Z}_m} =\prod_{h\in \mathbb{Z}_m} \left(\hat{\triangle}\prod_{k=1}^{m-1}\hat{\triangle}_{k,h}\right)\mathcal{D}_{X/\mathbb{Z}_m}.\label{atw}
\end{align} 

  Let us denote by $\mathcal{D}_{\mathcal{A,B}}^{X/\mathbb{Z}_m}$ the potential twisted by the class (\ref{typeb}). Then according to \cite{to1}, Corollary 6.2 the potential 
 $\mathcal{D}_{\mathcal{A,B}}^{X/\mathbb{Z}_m}$ differs from $\mathcal{D}_{\mathcal{A}}^{X/\mathbb{Z}_m}$ by the translation
 \begin{align*} 
  \mathbf{t}(z) \mapsto \mathbf{t}(z) +z - z \operatorname{Td}(-\mathbf{L}_z^{-1})\prod_{k=1}^{m-1} \operatorname{Td}_{\zeta^k} (-\mathbb{C}_{\zeta^k} \mathbf{L}_z^{-1}) =\\
                    =  \mathbf{t}(z) +z + z \frac{1-e^{z}}{z }\prod_{k=1}^{m-1}(1-\zeta^k e^z) = \mathbf{t}(z) +z + (1-e^{mz}).
      \end{align*} 
  In other words
  \begin{align}
  \mathcal{D}_{\mathcal{A,B}}^{X/\mathbb{Z}_m}\left(\sum_h \mathbf{t}_h h - z\right) = \mathcal{D}_\mathcal{A}^{X/\mathbb{Z}_m}\left(\sum_h \mathbf{t}_h h +1 -e^{mz}\right).\label{btw}
  \end{align}

  We use again the results of \cite{to1} to express the potential twisted by the nodal classes (\ref{typec}) in terms of $\mathcal{D}_{\mathcal{A,B}}^{X/\mathbb{Z}_m}$ . The smoothing conormal direction to a nodal stratum $\mathcal{Z}_d$ is given by $L_+ \otimes L_-$ where $L_+, L_-$ are cotangent line bundles to the branches with first Chern classes $\psi_+, \psi_-$. Define the coefficients $A_{a,b}^d$  as  
  \begin{align}
   \sum_{a,b \geq 0} A^d_{a,b}\overline{\psi}^a \overline{\psi}^b & = \frac{1-\operatorname{Td}^\vee (L_+L_- -1)\prod_{k=1}^{m/d-1}\operatorname{Td}^\vee_{\zeta^k}(L_+L_- -1)}{\psi_+ +\psi_-}  \nonumber\\
  & = \frac{d}{(\overline{\psi}_+ + \overline{\psi}_-)}-\frac{m/d}{e^{m/d^2(\overline{\psi}_+ + \overline{\psi}_-)}-1}.
  \end{align} 
 Here we denoted by $\overline{\psi}_+ = d\psi_+ , \overline{\psi}_- =d\psi_-$   the first Chern classes of cotangent line bundles on the cover curve.
 Let us define
   \begin{align*}
    \nabla_h := exp\left(\sum_{a,b,\alpha,\beta} \frac{\hbar}{2} A^d_{a,b}\partial_{a,\alpha}g^{\alpha\beta}\partial_{b,\beta}\right),
    \end{align*}

  where $d=ord(h)$ and $\alpha,\beta$ index elements of $H^*(X)$, the summation is after a basis $\{\varphi_\alpha\}$ of $H^*(X)$ and $(g^{\alpha\beta})$ is the matrix inverse to 
 \begin{align*}
 g_{\alpha\beta} =\int_X \varphi_\alpha\smile \varphi_\beta .
 \end{align*}
 Then Theorem 1.3 of \cite{to1} states that
   \begin{align}
   \mathcal{D}_{tw}^{X/\mathbb{Z}_m}     = \prod_{h\in \mathbb{Z}_m} \nabla_h\mathcal{D}_{\mathcal{A,B}}^{X/\mathbb{Z}_m}.\label{ctw}
        \end{align}
 
 The sequence of formulae (\ref{xxzm}),(\ref{atw}),(\ref{btw}),(\ref{ctw}) express the  potential of $\mathcal{D}_{tw}^{X/\mathbb{Z}_m}$ in terms of Coh GW potential of $X$.  
 
\section{ The K-theoretic GW potential of $X$}

 We first compute  the trace of the action of $\gamma$ on the partial ramification points of a quotient curve $C/\mathbb{Z}_m$.  
 
 \begin{acg}\label{011}
{\em  Assume that a special point on the quotient curve $C/\mathbb{Z}_m$ is a branch point of type $r$. Let $d=g.c.d.(r,m)$ and let $a,b$ be such $d=ar+mb$. Assume each of the $d$ preimages of the point on $C$ has input $\mathbf{t}(L)$. Then the input of that point in the correlators in KRR (viewed as integrals on the moduli spaces of maps to $X/\mathbb{Z}_m$) is $ \operatorname{ch}\Psi ^d \mathbf{t}(e^{2\pi i ad/m}L^{d/m})$.} 
 \end{acg} 
 
 Proof: The $\Psi^d$ occurs because there are $d$ preimages of the point on the cover curve $C$ which are identified on the quotient. The generator of the group $\mathbb{Z}_{m/d} $ acts on the cotangent line at each preimage by $e^{2\pi i ad/m}$. The extreme cases $r=1, r=m$ have been covered in \cite{gito}). 
 
 \begin{acg1}
 {\em  If the branch point is not a marked point is tantamount to considering it as a marked point with the input $\mathbf{t}(L) =1 - L$ (see \cite{gito}). Hence the contributions in the correlators in KRR is $ \operatorname{ch}\Psi ^d (1- e^{2\pi i ad/m}L^{d/m})$.}
 \end{acg1}
 
  According to KRR theorem the K-theoretic GW potential $\mathcal{D}^K_X (\mathbf{t}(L))$ of $X$  is (the exponential of) a sum over graphs $\Gamma$ of cohomological integrals. We denote the contributions of vertices $v$ of $\Gamma$ by $(\mathcal{D}_{tw}^{X/\mathbb{Z}_{m_v}})_v$: they are certain correlators from the theory $\mathcal{D}^{X/\mathbb{Z}_{m_v}}_{tw}$ where:

\begin{itemize}
\item the insertion in the correlators at each seat corresponding to a branch point of type $r$ which is a special (i.e. marked or node) point on the cover curve is $\hbar^{(m_v - d)} \operatorname{ch}\Psi ^d \mathbf{t}( e^{2\pi i ad/m}L^{d/m})$. Here $a,d$ are determined by the ramification datum as in Remark \ref{011}.
\item  the insertion at each branch point of type $r$ which is not a special point on the marked curve is $\hbar^{(m_v-d)} \operatorname{ch}\Psi ^d (1-e^{2\pi i ad/m} L^{d/m})$. 
\item we make the change of variables $Q^{d_v}\mapsto Q^{d_v m_v}$, $\hbar\mapsto \hbar^{m_v} $.
\end{itemize}
\begin{rem}
{\em The presence of $\hbar$ in the correlator insertion at marked points is to account for the difference in genus between $C_v$ and $C_v/\mathbb{Z}_{m_v}$.}
\end{rem}

Let now $e$ be an edge of $\Gamma$ adjacent to vertices $v$ and $w$, $F_v=(v,e), F_w=(w,e)$. At most one of the two quotient curves meeting at the node can has a disconnected cover. Assume $C_w$ consists of a disjoint union of $d$ curves. Then the contribution in KRR given by the normal directions which smoothen the nodes equal
\begin{align*}
\frac{1}{1-\zeta_{F_v} L_v^{1/m_v}\zeta_{F_w} \Psi^d L_w^{d/m_w}},
\end{align*}
where $\zeta_{F_v}, \zeta_{F_w}$ are primitive roots of orders $m_v$ and $m_w/d$ respectively. This class gets distributed in the correlator seats corresponding to the node on both branches.   

Let $\{\Phi_a\}, \{\Phi^a\}$ be dual bases in $K^0(X)$ 
and define the differential operator $A_e$ in the following way: expand the expression
\begin{align}
\frac{\hbar}{2}\left(\sum_a \operatorname{ch} \frac{\Phi_a\otimes \Phi^a}{1-\zeta_{F_v}L_v^{1/m_v} \zeta_{F_w}L_w^{d/m_w}} \right),\label{edgeop}
\end{align}
 as
 \begin{align*}
\frac{\hbar}{2} \left(\sum_{a,b, m,n} A_{a,m;b,n}\varphi_a\overline{psi}_v^m \varphi_b \overline{\psi}_w^n \right) \in H^*(X)[[\overline{\psi}_w]]\otimes H^*(X)[[\overline{\psi}_v]].
 \end{align*}
 Then
 \begin{align*}
 A_e:= \Psi^d \left[\frac{\hbar}{2} \left(\sum_{a,b, m,n} A_{a,m;b,n} \partial_{a,m}\partial_{b,n} \right)\right].
 \end{align*}
where $\Psi^d$ acts on $\varphi_b, \overline{\psi}_w$ ,$\hbar$ (both in the expression for $A_e$ and in $(\mathcal{D}_{tw}^{X/\mathbb{Z}_{m_w}})_w$)  and  $Q$ in $(\mathcal{D}_{tw}^{X/\mathbb{Z}_{m_w}})_w$.

We can now state
\begin{mainst}\label{mainthm}

The permutation-equivariant K-theoretic GW potential of $X$ equals:
\begin{align*}
log(\mathcal{D}_X^K ) = \sum_\Gamma  \prod_{e}  A_e \left( \prod_{v\in
  V(\Gamma)} \frac{(\mathcal{D}_{tw}^{X/\mathbb{Z}_{m_v}})_v}{\vert S_v\vert}\right).
\end{align*}

\end{mainst}

\begin{kcoh}
{\em All permutation-equivariant K-theoretic GW invariants of $X$ are determined by the cohomological GW invariants.}
\end{kcoh}
 
 
 
 
 
 
\end{document}